\documentclass[reqno,12pt]{amsart}

\usepackage{amsmath,amssymb,amsthm}
\usepackage{graphicx}
\usepackage{psfrag}
\usepackage{hyperref}
\usepackage[dvipsnames,usenames]{color}
\usepackage{orcidlink}

\newtheorem{guia}{}
\newtheorem{rem}{}

\newtheorem{teorema}[guia]{Theorem}

\newtheorem{lema}[guia]{Lemma}

\newtheorem{obs}[rem]{\it Remark}
\newtheorem{obss}[rem]{\it Remarks}
\newtheorem{ejem}[guia]{Example}

\newcommand{\al}{\alpha}

\newcommand{\De}{\Delta}

\newcommand{\e}{\varepsilon}

\newcommand{\g}{\gamma}

\newcommand{\na}{\nabla}
\newcommand{\om}{\omega}
\newcommand{\Om}{\Omega}
\newcommand{\Omb}{\overline{\Omega}}
\newcommand{\p}{\partial}

\newcommand{\R}{\mathbb R}
\newcommand{\s}{\sigma}

\newcommand{\vf}{\varphi}

\newcommand{\z}{\zeta}

\def\iz{\left }
\def\der{\right }

\def\menos{\setminus}

\newcommand\map[5]
{\begin{array}{cccl}
{#1}:&\ {#2}\ &\longrightarrow&\ {#3}\\
&\ {#4}\ &\longmapsto&\ {#5}
\end{array}}

\newcommand\bla[1]{{\mathbf {#1}}}

\def\sob{W^{1,p}}

\newcommand\uno[1]{{\{{#1}\}}}
\def\menos{\setminus}

\def\eren{{\mathbb R}^N}

\def\ene{{\mathcal N}}

\def\lae{\eqref}

\newcommand\rango[3]{{#1}\le{#2}\le{#3}}

\newcommand\lai[1]{\item[{#1})]}

\def\bn{{\mathbf n}}
\def\tij{{\widetilde J}}
\def\dist{\text{dist}}
\def\lau{{\mathcal U}}

\begin{document}

\title[Trace inequality in $BV$]
{\bf Trace inequality for $BV(\Omega)$ in smooth domains}

\author[J.C. Sabina de Lis]{Jos\'e C. Sabina de Lis\, \orcidlink{0000-0002-2378-7614}}

\date{\today}

\address{Jos\'{e} Sabina de Lis
\hfill \break \indent Departamento de An\'{a}lisis Matem\'{a}tico {\rm and} IUEA,
\hfill \break \indent Universidad de La
Laguna, \hfill \break \indent P. O. Box 456, 38200 -- La Laguna, SPAIN.}
\email{{\tt josabina@ull.edu.es}}



%%%%%%%%0.Abstract

\begin{abstract} This work is devoted to prove an optimum version of
the trace inequality associated to the embedding $BV(\Om)\subset L^1(\p\Om)$.
Special emphasis is placed on the regularity that the domain $\Om$
should exhibit for this result to be valid.
\end{abstract}

\maketitle

\tableofcontents
%%%%%%%%%%%%%1.Introducci\'{o}n

\section{Introduction}\label{s:1}
It has been known for a long time that an inequality of the form,
\begin{equation}
\label{i1}
\int_{\p\Om}|u| \le C_1\int_\Om |Du|+ C_2\int_\Om |u|,
\end{equation}
$C_i >0$, $i=1,2$, is satisfied for the functions $u$ belonging to the space
$W^{1,1}(\Om)$ provided $\Om\subset \eren$ is a bounded Lipschitz domain (see the pioneering work
\cite{Gag}, also \cite{Ne},
\cite{Tar} and \cite{Ad}, the later for more regular domains). It expresses the continuity of the embedding
$W^{1,1}(\Om)\subset L^1(\p\Om)$ while $u$ in the surface
integral has the status of the `trace' (the restriction) of the function on the boundary $\p\Om$.
Moreover it follows from a standard approximation argument (see Theorem \ref{t1} below)
that \lae{i1} extends to the space $BV(\Om)$ (Section \ref{s:2}) with exactly the same values
of the constants $C_i$. In this context, a more ambitious version of \lae{i1} was stated
in \cite[Sec. 3]{AnG} by allowing domains $\Om$ whose boundary satisfies $\mathcal H_{N-1}(\p\Om)<\infty$,
$\mathcal H_{N-1}$ standing for the $N-1$ dimensional Hausdorff measure, and exhibits a suitable local behavior
(see \cite{AFP} for a more recent account).

With respect to the constant $C_1$ in \lae{i1}, a preliminary analysis reveals that it must satisfy the restriction
$C_1\ge 1$ (Section \ref{s:2}). This fact highly contrasts with the corresponding inequality for the exponent $p>1$,
$$
\int_{\p\Om}|u|^p \le \e \int_\Om |Du|^p +C_\e \int_\Om |u|^p,
$$
which holds for every $u\in \sob(\Om)$, and where $\e>0$ can be taken as small as desired.

Moreover, provided that $\Om$ is $C^1$ then it can be shown that (\cite{Mot} and Section \ref{s:2} below),
\begin{equation}
\label{i2}
\inf{C_1} = 1,
\end{equation}
where the infimum is extended to all constants $C_1>0$ such that inequality \lae{i1} holds for some $C_2>0$.
A similar statement in this direction had been earlier suggested  in \cite{AnG} by considering a wider class
of domains encompassing the Lipschitzian ones. In this case unity in \lae{i2} should be replaced by a larger
minimizing constant (see Example \ref{eje0} below).

It is therefore a natural question to decide if the better possible value $C_1 = 1$ could be placed in \lae{i1}.
This paper is devoted to elucidate the optimum smoothness requirements on $\p \Om$ under which the answer is affirmative.
For the time being, we anticipate that it is not enough for
$\Om$ to be merely $C^1$ (not even $C^{1,\al}$ for every $0 < \al <1$) in order that this
feature to be true.

On the other hand, to clarify this technical subtlety turns out to be crucial for the validity of a
further important property. Namely, the lower semicontinuity in $BV(\Om)$ of the functional,
\begin{equation}
\label{i3}
J(u) = \int_\Om |Du| + \int_{\p\Om}F(x,u),
\end{equation}
with respect to $L^1(\Om)$. Provided that $F$ is Lipschitzian in $u$ with a constant
not greater than unity it was shown in \cite{Mod} that the property holds true (see Theorem \ref{t3} below). Most importantly, the
proof in \cite{Mod} requires
the validity of inequality \lae{i1} exactly with the critical value $C_1=1$. Author just reports (\cite[p. 493]{Mod}):
``{\it Let us remark that $C_1=1$ because $\p\Om$ is smooth}''.
He quotes \cite{AnG} as a source where however there is no any
reference to this precise concern (see also Remark \ref{o:5} below).

The importance of the lower semicontinuity of $J$ stated in  \cite{Mod}
becomes evident when dealing with boundary value problems involving the
one--Laplacian operator $\De_1 u = \text{\rm div}\iz(\dfrac{Du}{|Du|}\der)$. Specially,
when flow--type conditions are imposed on $\p\Om$. See for instance \cite{DNOT}, \cite{MRS},
\cite{AMMo}, \cite{HauM}, \cite{LaSe}, \cite{DOSe}, \cite{SS_24}, \cite{BDP}.
In most of these works such a property is
invoked under the assumption ``$\Om$ is smooth''. Nevertheless, it is widely
believed that the semicontinuity of $J$ still holds
true if $\Om$ is just $C^1$. Thus, it should be desirable to get an insight on the
precise smoothness
degree required to $\p\Om$ for this property to be fulfilled. As a matter of fact, it will
be also seen in Section \ref{s:4} that class $C^1$ is not enough.

The work is organized as follows. Section \ref{s:21} recalls the background on the space $BV(\Om)$
needed in this work. The geometry of the domain $\Om$ near the boundary $\p\Om$ is most relevant
when addressing an optimum version of \lae{i1}. It is discussed in Section \ref{s:22} where
a suitable coordinate system defined in a neighborhood of $\p\Om$ is introduced. We also review in Sections \ref{s:23}
and \ref{s:24} the specific results of \cite{Mot} and \cite{Mod} concerned with our study here. Main result
Theorem \ref{t1} is stated in Section \ref{s:3}. Some examples analyzing the smoothness of
the domain needed for the validity of
several of the discussed results are presented in Section \ref{s:4}.

\section{Basic and technical previous results}\label{s:2}

\subsection{Functions of bounded variation}\label{s:21}
Let $\Om\subset \eren$ be a bounded domain where $L^1(\Om)$ designates the standard space of Lebesgue
integrable functions. As a proper subspace, $BV(\Om)\subset L^1(\Om)$ stands for the functions
of bounded variation. Namely, those $u\in L^1(\Om)$ whose gradient $Du$ in distributional sense defines
a Radon measure. By designating as $\|Du\|$ its associated `total variation' measure it is also
required that $\|Du\|(\Om) < \infty$. Such a  number is often represented
as,
$$
\|Du\|(\Om) = \int_\Om |Du|.
$$
Space $BV(\Om)$ becomes complete (Banach) when endowed with the norm
$\|u\|_{BV(\Om)} = \|u\|_{L^1(\Om)}+\|Du\|(\Om)$.
As a  further concept linked to the theory, it is said that $E\subset \Om$ is a set of finite perimeter in $\Om$ if
$\chi_E\in BV(\Om)$, $\chi_E$ designating the characteristic function of $E$.
In this case $\|D\chi_E\|(\Om)$ defines the relative perimeter $P(E,\Om)$ of $E$ with
respect to $\Om$. Let just remark to clarify  that provided $Q\subset \eren$ is
a Lispchitz domain with $Q\cap\Om$ nonempty then $P(Q\cap \Om,\Om)$ is given by
$\mathcal H_{N-1}(\p Q\cap \Om)$. This quantity is
also denoted as $|\p Q\cap \Om|_{N-1}$.

By assuming that $\Om$  is a Lipschitz domain then $BV(\Om)$ possess a trace operator
$\g:BV(\Om)\rightarrow L^1(\p\Om)$ which coincides with the restriction of functions
to $\p\Om$ when observed in $C^1(\Omb)$. By an abuse of notation $u$ will be used instead of $\g(u)$ to mean
the trace of $u$ in $\p\Om$. The trace operator is continuous when convergence
in $BV(\Om)$ is observed in the weaker sense of the {\it strict topology}. Namely, $u_n\to u$
strictly in $BV(\Om)$ if $u_n\to u$ in $L^1(\Om)$ while $\|Du_n\|(\Om)\to \|Du\|(\Om)$ (\cite{AFP}).
Strict topology turns out to be quite convenient when dealing with approximation affairs in $BV(\Om)$.
On the other hand, inequality \lae{i1} extends easily to $BV(\Om)$ in the way described in the
proof of Theorem \ref{t1} below.

\subsection{Tubular neighborhoods}\label{s:22}

The analysis in Section \ref{s:3} makes use of a smart coordinate system defined near $\p \Om$.
It is formed through a suitable family of normal segments which fills the space surrounding the
boundary of $\Om$. To this aim we are introducing some few definitions and corresponding
properties.

Assume that $\Om$ is a bounded domain such that $\p \Om$ is a compact $C^1$ and  $N-1$ dimensional manifold.
Then it is possible to define a unit normal field $\bn$ attached to $\p\Om$ which is directed towards $\Om$.
We are assuming that $\p\Om$ is oriented with such a field $\bn$.

Any $x\in \eren$ has at least a closer point $y\in \p\Om$ which means that $d(x)=|x-y|$ where,
$$
d(x)=\dist(x,\p\Om)=\inf_{y'\in \p\Om} |x-y'|.
$$
To every $y\in \p\Om$ a number $r(y)$, the {\it reach} of $y$ (\cite{Fe}), is defined as,
\begin{multline*}
r(y) = \\
\sup\{r>0:\, \text{$\forall x\in B(y,r)$ a unique
 $z\in\p\Om$ exists such}\\ \text{that $d(x)=|x-z|$}\}.
\end{multline*}
Geometrically, $r(y)>0$ entails the existence of an inner and an outer tangent ball to $\p\Om$ at $y$,
$\overline{B}\menos\uno{y}\subset \Om$ and $\overline{B'}\menos\uno{y}\subset \eren\menos\Omb$ respectively,
both of them with a radius $0 < r < r(y)$. Thus $r(y)^{-1}$ may be regarded as a sort of generalized
curvature of $\p\Om$ at $y$ while $r(y)$ has the status of a  curvature radius.

It follows from the implicit function theorem that $r(y)>0$ on $\p\Om$ provided it is a $C^2$ manifold.
On the contrary, $r$ could possibly vanish
somewhere in $\p\Om$ if it is only of class $C^{1,\al}$ for $0 <\al<1$. An example of this instance is presented
in Remark \ref{o:6}. That is why the next statement may be regarded as {\it optimal}.

\begin{teorema}
\label{t0}
Assume $\Om\subset \eren$ is a bounded $C^1$ domain, i. e. $\p\Om$ is a compact and $C^1$ $N-1$ dimensional
manifold. Then, the following assertions are equivalent:

\lai a $r(y)>0$ for every $y\in\p\Om$,

\lai b $\p\Om$ is $C^{1,1}$, i. e. $\Om$ is a class $C^{1,1}$ bounded domain.
\end{teorema}

That b) is sufficient for a) is shown either in \cite[Lem. 4.11]{Fe} or \cite[Lem. 2.1]{Luc}.
For the necessity of b) reader is referred to \cite[Th. 4.1]{Luc} (cf. also \cite[Cor. 2]{LeP}).

It can be further shown that function $r(y)$ is continuous. Since $\p\Om$ is compact it follows that,
$$
r_{\p\Om}:= \inf_{y\in \p\Om}r(y)>0,
$$
whenever option a) holds.  Moreover,
$$
r_{\p\Om}^{-1} = \sup_{x,y\in\p\Om\,, x\neq y}\frac{|\bn(x)-\bn(y)|}{|x-y|}.
$$
If this is the case,  a projection mapping,
$$
\pi:\lau \longrightarrow \p\Om,
$$
is defined such that $|x-\pi(x)|= d(x)$ for all $x\in \lau$, where $\lau = \uno{x:d(x)<r_{\p\Om}}$.
The more important features linking the previous definitions are gathered in
the following statement.

\begin{teorema}
\label{t00} Let $\Om\subset \eren$ be a bounded $C^{1,1}$ domain, i. e.
$\p\Om$ constitutes a $C^{1,1}$ and compact $N-1$ dimensional manifold which is
oriented with the inward unit normal $\bla n$. Then,

\lai i $\pi \in C^{0,1}(\lau,\eren)$.

\lai{ii} For each $y\in \p\Om$, $\pi^{-1}(y) = \uno{y+t \bla n(y):\ |t|<r_{\p\Om}}$.

\lai{iii} If $\ene = \uno{(p,v):\, p\in\p\Om\, , v\in \text{\rm span}\,\uno{\bla n(p)}}$ stands
for the normal bundle to $\p\Om$ then the mapping
$$
\map{H}{\lau}{\ene}x{(y,v)=(\pi(x),x-\pi(x)),}
$$
defines a Lipschitz homeomorphism onto its image, with Lipschitz inverse
defined as $H^{-1}(y,v)=y+v$.

\lai{iv} The {\rm signed} distance,
$$
\tilde d(x) = (x-\pi(x))\cdot \bla n(\pi(x))
=
\begin{cases}
\, d(x),&\qquad x\in \Omb,\\
-d(x),&\qquad x\in\eren\menos\Omb,
\end{cases}
$$
is $C^1$ in $\lau$.
\end{teorema}

\begin{proof}
 Assertions i), ii) and iii) are consequence of items 8), 2) and 13) in \cite[Th. 4.8]{Fe}
 while the fact that $\tilde d$ is $C^1$ in  $\lau$ is the conclusion of \cite[Th. 1]{Krap}
 which is in fact the main achievement of this reference.
\end{proof}

Theorem \ref{t00} implies that $\lau$ becomes a tubular neighborhood of $\p\Om$ (\cite{GuiP}) where
the coordinates of a point $x$ are $(\pi(x),x-\pi(x))$. Alternatively, they can be expressed in the
more manageable form $(\pi(x),\tilde d(x))$ where vector $x-\pi(x)$ is replaced
by the number $\tilde d(x)$.

Let now $\uno{(g_i,U_i)}$ be a finite atlas  for $\p\Om$. This means that for
every $\rango 1im$ the set $U_i\subset \R^{N-1}$
is open, $g_i\in C^{1,1}(U_i,\eren)$ is one to one and $\text{rank}\, Dg_i(s)=N-1$ for $s\in U_i$.
Finally,
$$
\p \Om = \cup_{i=1}^m g_i(U_i).
$$
Define,
\begin{equation}\label{e40}
\map{h_i}{U_i\times (-r_{\p\Om},r_{\p\Om})\subset \R^{N}}\eren{(s,t)}{h_i(s,t)= g_i(s)+t\bn(g(s_i)).}
\end{equation}
Then $h_i$ constitutes a Lipschitz homeomorphism from $U_i\times (-r_{\p\Om},r_{\p\Om})$
onto its image $V_i\subset \lau$, with a Lipschitz inverse $h_i^{-1}$. Furthermore,
$$
\lau = \cup_{i=1}^m V_i.
$$
Therefore, pair $(s,t)$ represents the local expression in $V_i$ of the tubular coordinates $(\pi(x),\tilde d(x))$
introduced above. In fact, $g_i(s)=\pi(x)$, $t = \tilde d(x)$ for $x\in V_i$.

\subsection{Optimum constant in \lae{i1}}\label{s:23}

We begin by observing a basic estimate for $C_1$. In the next result $\Om$ is only assumed
to be Lipschitzian.

\begin{lema}\label{l1}
Assume that $\Om$ is a bounded Lipschitz domain. Then, constant $C_1$ in the trace inequality \lae{i1}
satisfies the lower estimate $C_1\ge 1$.
\end{lema}

\begin{proof}
If $\Om$ is Lipschitizan then there exists an increasing sequence
$\Om_n\subset\subset \Om_{n+1}\subset\subset \Om$ of $C^\infty$ strict
subdomains of $\Om$ such that $\chi_{\Om_n}\to \chi_\Om$ in $L^1(\Om)$,
while,
$$
P(\Om_n,\eren)\to P(\Om,\eren).
$$
See for instance \cite[Th. 1]{MPe}. Define $u = u_n = \chi_{Q_n}$,
$Q_n = \Om\setminus \Omb_n$, $\chi_{Q_n}$ meaning the characteristic
function of $Q_n$. By inserting $u_n$ in \lae{i1} we get,
$$
\mathcal H_{N-1}(\p \Om)\le C_1 \int_\Om |D\chi_{Q_n}|+ C_2 |Q_n| = C_1 P(Q_n,\Om) + C_2|Q_n|,
$$
where $|Q_n|$ designates the Lebesgue measure of $Q_n$. Then,
$$
\mathcal H_{N-1}(\p \Om)\le C_1 P(\Om_n,\eren)+ C_2|Q_n|.
$$
We achieve $C_1\ge 1$ after taking limits.
\end{proof}

\begin{obss}
\rm
Estimate in Lemma 1 is obtained in \cite{Mot} when $\Om$ is a smooth domain.
For the sake of completeness, it is shown in Section \ref{s:3} that an argument as the one in
Lemma \ref{l1} can be carried out with more regular functions $u_n$ provided
$\Om$ is $C^1$.
\end{obss}

Next statement asserts that $C_1$ can be chosen as close to unity as desired when $\Om$
is $C^1$ (\cite[Th. 1.7]{Mot}). For the interested reader, an independent proof of the
result can be given with the tools of Section \ref{s:3} and is therefore presented
there.

\begin{teorema}\label{t2}
Assume that $\Om$ is a bounded $C^1$ domain. Then, for every $\e>0$ there exists $C_2= C_2(\e)>0$
such that,
\begin{equation}
\label{2:1}
\int_{\p\Om}|u| \le (1+\e)\int_\Om |Du|+ C_2\int_\Om |u|,
\end{equation}
holds for all $u\in W^{1,1}(\Om)$.
\end{teorema}

\begin{obs}
\rm
As it appears in \cite{Mot}, Theorem \ref{t2} is not properly stated there
since it is asserted to be also valid for {\it piecewise} $C^1$ domains.
The proof argued in this work can only be valid if $\Om$ is $C^1$.
Indeed, for a piecewise $C^1$ but Lipschitz domain,
Example \ref{eje0} below shows that the optimum constant $C_1$ is bounded away from
unity.
\end{obs}

We now analyze an specific case where the best value $C_1=1$ is achieved. Namely, the
situation where $\Om$ is either a ball or an annulus. It is already addressed in \cite{Mot}
in the space $W^{1,1}(\Om)$.
A direct proof in the spirit of our general result in Section \ref{s:3} is next included
for completeness.

\begin{lema}
\label{l2} Let $\Om$ be either of the radially symmetric domains $B:=\uno{|x|<b}$ (a ball) or $A:=\uno{a<|x|<b}$
(an annulus). Then, the inequality \lae{i1} adopts the form,
\begin{equation}
\label{2:2}
\int_{\p\Om}|u| \le \int_\Om |Du|+ \frac{Na^{N-1}+Nb^{N-1}}{b^N-a^N}\int_\Om |u|, \qquad \text{for all $u\in BV(\Om)$,}
\end{equation}
where it is set $a=0$ in the case of the ball.
\end{lema}

\begin{proof}
By the reasons explained in the forthcoming Theorem \ref{t1} it is enough with considering
$u\in C^1(\Omb)$. In addition, the proof in the case of the annulus essentially encompasses
the one of the ball. Thus we only deal with $\Om = \uno{a<|x|<b}$ where $0 < a < b$.

By employing polar coordinates we write:
$$
v(r,y) = u(ry),\qquad r\ge 0\,,\quad |y|=1.
$$
Thus,
$$
-v(a,y) = -v(r,y) + \int_a^r v'_r(t,y)\ dt,
$$
what implies,
$$
|v(a,y)| \le |v(r,y)| + \int_a^r |v'_r(t,y)|\ dt.
$$
Then, we get by integration,
\begin{multline*}
\int_{|y|=1}\int_a^b |v(a,y)| r^{N-1}\, dr\, dS_y \le \int_{|y|=1}\int_a^b |v(r,y)| r^{N-1}\, dr\, dS_y \\
+ \int_{|y|=1}\int_a^b \iz ( \int_a^r |v'_r(t,y)|\ dt \der) r^{N-1}\, dr\, dS_y,
\end{multline*}
$dS_y$ designating the surface area element on the unitary sphere $\uno{|y|=1}$.
By changing the integration order in the last integral we arrive to,
\begin{multline}\label{2:3}
\frac{b^N-a^N}{Na^{N-1}}\int_{|x|=a}|u| \le \int_{\Om}|u|
%\\
+ \frac 1N \int_{|y|=1}\int_a^b \iz ( b^N-t^N \der) |v'_r(t,y)|\, dt\, dS_y.
\end{multline}
A symmetric argument starting at,
$$
v(b,y) = v(r,y) + \int_r^b v'_r(t,y)\ dt,
$$
leads to the inequality,
\begin{multline}\label{2:4}
\frac{b^N-a^N}{Nb^{N-1}}\int_{|x|=b}|u| \le \int_{\Om}|u|
%\\
+ \frac 1N \int_{|y|=1}\int_a^b \iz ( t^N-a^N \der) |v'_r(t,y)|\, dt\, dS_y.
\end{multline}
By adding the inequalities \lae{2:3} and \lae{2:4} we obtain,
\begin{multline}\label{2:5}
\int_{\p\Om}|u| \le \frac{Na^{N-1}+Nb^{N-1}}{b^N-a^N}\int_\Om |u| \\
+ \int_{|y|=1}\int_a^b \iz (\frac{a^{N-1}(b^N-t^N)+b^{N-1}(t^N-a^N)} {b^N-a^N} \der)|v'_r(t,y)|\, dt\, dS_y.
\end{multline}
It is claimed now that,
\begin{equation}
\label{2:6}
\frac{a^{N-1}(b^N-t^N)+b^{N-1}(t^N-a^N)} {b^N-a^N}\le t^{N-1},
\end{equation}
for all $\rango atb$. Then, estimate \lae{2:2} follows from \lae{2:5}, \lae{2:6} and the fact that
$$
|v'_r(t,y)|= |Du(ry)|\,, \qquad \rango arb\,,\quad |y|=1.
$$
To show the claim define the function $\vf(t)$ as the difference: ``$t^{N-1}$ minus the left hand side of
\lae{2:6}''. Then $\vf$ vanishes at $t=a,b$, has a  unique critical point $a<t_c< b$,
$$
t_c = \frac{N-1}N \frac{b^N-a^N}{b^{N-1}-a^{N-1}},
$$
while $\vf'(a)>0$ and $\vf'(b)<0$. All these facts together prove the claim.
\end{proof}

\begin{obs}
\rm By taking $u=1$ in the trace inequality \lae{i1} it is found that constant $C_2$ must satisfy the lower estimate,
$$
C_2 \ge \dfrac{\mathcal H_{N-1}(\p\Om)|}{|\Om|}.
$$
A nice feature of \lae{2:2} is that the optimum values both for $C_1$ and $C_2$ are simultaneously attained
in radially symmetric domains.
\end{obs}

\subsection{A semicontinuity result}\label{s:24}

The following statement corresponds to \cite[Prop. 1.2]{Mod}.

\begin{teorema}\label{t3}
Let $\Om\subset\eren$ be a bounded {\rm smooth}  domain, $F:\p\Om\times \R\to \R$ a Carath\'{e}dory  function
satisfying,
$$
|F(x,u)-F(x,v)|\le |u-v|,
$$
for all $u,v\in \R$ and $\mathcal H_{N-1}$--almost all $x\in \p\Om$. Then
the functional $J$ defined in \lae{i3} is lower semicontinuous in $BV(\Om)$ with respect to the topology of $L^1(\Om)$.
\end{teorema}

\begin{obss}
\rm
\lai 1 An explicit reference on the regularity degree of $\Om$ for Theorem \ref{t3} to be true is not provided in \cite{Mod}.

\lai 2 Proof of Theorem \ref{t3} strongly relies upon the possibility of choosing $C_1=1$ in \lae{i1}. Since we are showing
in next section that this feature is true provided $\Om$ is $C^{1,1}$ one wonders whether the result still remains valid
for less smooth domains. Example \ref{e:mod} in Section \ref{s:4} shows that class  $C^{1,\al}$, with $0 < \al < 1$, does
not suffice.
\end{obss}

\section{The trace inequality}\label{s:3}

We are next stating the main result. Recall that the key point is the possibility of taking the unity
as the constant multiplying the second integral in \lae{e:ppal}.

\begin{teorema}\label{t1}
Let $\Om\subset \eren$ be a bounded $C^{1,1}$  domain.
Then, there exists a constant $C$ only depending
on $\Om$ such that the inequality,
\begin{equation}
\label{e:ppal}
\int_{\p\Om}|u| \le \int_\Om |Du| + C\int_\Om |u|,
\end{equation}
holds for each $u\in BV(\Om)$.
\end{teorema}

\begin{proof}
Thanks to the version of  Meyer--Serrin theorem extended to $BV(\Om)$ (\cite[Th. 2.2.2]{EG}) an arbitrary function $u\in BV(\Om)$
can be approached in the strict topology of $BV$ (see Section \ref{s:21}) by a sequence of functions $u_n\in BV(\Om)\cap C^\infty(\Om)$.
This means by definition that,
$$
\int_\Om |u-u_n|\to 0\qquad \text{and} \qquad \int_\Om |Du_n|\to \int_\Om |Du|.
$$
The continuity of the trace operator with respect to the strict topology (\cite{AFP}) then also implies that,
$$
\lim \int_{\p\Om}|u-u_n|=0.
$$
On the other hand, since a smooth domain as $\Om$ satisfies either the `segment property' (\cite{Ad}) or alternatively the `graph
property' (\cite{Tar}), then any function $u\in W^{1,1}(\Om)$ can be approached in this space by a sequence of functions in $C^\infty(\Omb)$
(by definition the restrictions to $\Omb$ of functions in $C^\infty(\eren)$). Therefore, to show \lae{e:ppal} it can be assumed
from the start that $u\in C^1(\Omb)$.

We are first assuming that $\Om$ is $C^2$. Accordingly, we select a $C^2$ family
$\uno{(g_i,U_i)}_{i=1,\dots,m}$ of charts for $\p\Om$, exhibiting the properties detailed in Section \ref{s:22}.
These  charts can be furthermore arranged such that,
$$
\bn_{|g_i(U_i)} = \frac{{g_i}_{s_1}\wedge \cdots \wedge {g_i}_{s_{N-1}}}{|{g_i}_{s_1}\wedge \cdots \wedge {g_i}_{s_{N-1}}|},
$$
for all $\rango 1im$. In this expression ${g_i}_{s_k}$ stands for the partial derivative $\p_{s_k} g_i$ while the numerator
consists in the exterior product of all these derivatives.

By replacing $r_{\p\Om}$ in \lae{e40} by a possibly smaller number $\e>0$, mappings,
\begin{equation}
\label{e41}
\map{h_i}{U_i\times (-\e,\e)\subset \R^{N}}\eren{(s,t)}{h_i(s,t)= g_i(s)+t\bn(g(s_i))},
\end{equation}
now define $C^1$ diffeomorphisms onto their images,
$$
V_i =h_i(U_i\times (-\e,\e)).
$$
They satisfy,
$$
\Om\cap V_i = h_i(U_i\times (0,\e)),
$$
while,
$$
\lau_\e = \cup_{i=1}^m V_i\supset \p\Om,
$$
becomes tubular neighborhood of $\p\Om$ of size $\e$.
On the other hand, an open $V_0\subset \overline V_0\subset \Om$ can be found so that
$\uno{V_i}_{\rango 0im}$ constitutes a finite covering of $\Omb$.

To complete this preparatory phase of the proof we introduce a partition of unity
subordinated to the covering $\uno{V_i}_{\rango 0im}$ of $\Omb$. It consists on
a family $\uno{\vf_i}_{\rango 0im}$ of $C^\infty$ functions such that $\rango 0{\vf_i}1$
and every $\vf_i$ is supported in $V_i$. In addition,
$$
\sum_{i=0}^m \vf_i = 1.
$$
Thus, we can express $u\in C^1(\Omb)$ in the form,
$$
u = \sum_{i=0}^m \vf_i u=:\sum_{i=0}^m u_i.
$$
We now show that for every $\rango 1im$ the function $u_i$ satisfies an inequality of
the form,
\begin{equation}
\label{e2}
\int_{\p \Om}|u_i|\le \int_\Om |Du_i|+ C_i\int_\Om |u_i|,
\end{equation}
where the constant $C_i$ depends on the mean curvature of $\p\Om$. To prove \lae{e2}
let us write,
\begin{equation}
\label{e3}
\int_{\p \Om}|u_i|= \int_{\p \Om\cap V_i}|u_i|= \int_{U_i} |u_i(s,t)|J(s,0)\ ds.
\end{equation}
We are going to describe in more detail the coefficient $J(s,0)$ in the surface area element.
Here and by an abuse of notation we have represented,
$$
u_i(s,t) = u_i(h_i(s,t))=u_i(g_i(s)+t\bn(s)),\qquad 0 \le t<\e,
$$
where $\bn(s)$ stands for $\bn(g_i(s))$.
Remark now that,
$$
Dh_i(s,t) = \text{\rm col\,} (g_{s_1}+t\bn_{s_1},\dots,g_{s_{N-1}}+t \bn_{s_{N-1}},\bn),
$$
where subindex $i$ in $g$ is dropped to short, $g_{s_j} = \p_{s_j}g$ and
`{\it col\,}' is employed to describe the subsequent matrix by columns. If it is denoted,
$$
J(s,t)= \det Dh_i(s,t),
$$
then we observe that,
\begin{equation}\label{e33}
J(s,0) = \det \text{\rm col\,} (g_{s_1},\dots,g_{s_{N-1}},\bn) = |{g_i}_{s_1}\wedge \cdots \wedge {g_i}_{s_{N-1}}|.
\end{equation}

Thus $J(s,t)$ can be assumed to be positive for $\e$ small. Moreover, $J(s,0)\ ds$ is just the
expression of the surface area element on $\p\Om\cap V_i$.

On the other hand,
\begin{multline}\label{e30}
-u_i(s,0)J(s,0) = \int_0^{\e}\p_t(u_i(s,t)J(s,t))\ dt \\=
\int_0^{\e}\p_t u_i(s,t) J(s,t )\ dt + \int_0^{\e}u_i(s,t)\p_t J(s,t )\ dt .
\end{multline}
Thus,
\begin{multline}\label{e31}
|u_i(s,0)|J(s,0)\le \int_0^{\e}|\p_t u_i(s,t )| J(s,t )\ dt \\+
\int_0^{\e}|u_i(s,t )| \frac{|\p_t J(s,t )|}{J(s,t )} J(s,t )\ dt .
\end{multline}
By integrating in $U_i$ with respect to $s$ we obtain,
\begin{multline}\label{e32}
\int_{U_i} |u_i(s,0)|J(s,0)\ ds
\le \int_{U_i}\int_0^{\e}|\p_t u_i(s,t )| J(s,t )\ dt \ ds \\ +
\int_{U_i} \int_0^{\e}|u_i(s,t )| \frac{|\p_t J(s,t )|}{J(s,t )}J(s,t )\ dt \ ds.
\end{multline}
Since $J(s,0)\ ds$ defines the surface area element on $\p\Om\cap V_i$ we achieve the
inequality \lae{e2}. To this purpose, the change of variables $x=h_i(x,t)$ must be performed in the volumetric
integrals in the right hand side. Observe also that,
$$
\p_t u_i(s,t) = Du_i(h_i(s,t))\bn(s),
$$
and so $|\p_t u_i(s,t)|\le |Du_i|$. Finally,
$$
C_i = \sup_{U_i\times (0,\e)} \frac{|\p_t J|}{J}.
$$
In this regard it should be noticed that,
\begin{multline*}
\p_t J(s,0) =
\det \text{\rm col\,} (\bn_{s_1},\dots,g_{s_{N-1}},\bn) +\cdots+\det \text{\rm col\,} (g_{s_1},\dots,\bn_{s_{N-1}},\bn)
\\
= \text{\rm trace}\ \text{\rm col\,} (\bn_{s_1},\dots,\bn_{s_{N-1}}) J(s,0) =  \text{\rm trace\,} D\bn(s) J(s,0).
\end{multline*}
Thus,
$$
\frac{\p_t J(s,0)}{J(s,0)} = - \text{\rm trace\,} D\nu(g(s)) = (N-1)H(g(s)),
$$
where $H(x)$ stands for the mean curvature of $\p\Om$ when this surface is oriented with the {\it outer}
unit normal $\nu=-\bn$.
Reader is referred to \cite{Tho} for the geometrical background and to \cite[Sec. 3]{PPS}
for further computation details.
As a conclusion constant $C_i$ in \lae{e2}
is directly related to the mean curvature of $\p\Om$.

Starting at \lae{e2} we now obtain,
\begin{multline}\label{e4}
\int_{\p\Om}|u| = \sum_{i=1}^m \int_{\p\Om}\vf_i |u|= \sum_{i=1}^m \int_{\p\Om}|u_i|\\
\le \int_{\Om}|Du_0|+ \sum_{i=1}^m \int_{\Om}|Du_i| +\max\uno{\max_{\rango 1im}C_i,1}\int_\Om |u|.
\end{multline}
In addition,
\begin{multline}\label{e5}
\sum_{i=0}^m\int_\Om |Du_i|
%\\
\le \sum_{i=0}^m\int_\Om \vf_i |Du|+ \sum_{i=0}^m\int_\Om  |D\vf_i||u|\\
\le \int_\Om |Du| + (m+1) \max_{\rango 0im}\|D\vf_i\|_\infty\int_\Om |u|.
\end{multline}
By combining \lae{e4} and \lae{e5} we arrive to the desired inequality \lae{e:ppal}.

As for the case where $\Om$ is $C^{1,1}$, we choose a $C^{1,1}$ atlas
$\uno{(g_i,U_i)}_{\rango 1im}$ for $\p\Om$ with the properties
described in Section \ref{s:22}. Mappings $h_i$ in \lae{e41}, where $0 < \e<r_{\p\Om}$
is regarded as a parameter, define Lipschitz homeomorphisms with Lipschitz inverses
$h_i^{-1}$ onto their images $V_i$. Covering $\uno{V_i}_{\rango 0im}$ and associated
partition of unity $\uno{\vf_i}_{\rango 0im}$ are formed in the same way as before.

The key point to show \lae{e:ppal} consists alike in achieving the inequality \lae{e2}.
To this aim we first notice that by Rademacher's theorem $Dh_i(s,t)$ is defined
for all $|t|<\e$ and almost all $s\in U_i$. Moreover,
$$
J(s,t) = \text{det\,}Dh_i(s,t) = \sum_{k=0}^{N-1}J_k(s)t^k,
$$
hence $J$ is a polynomial in $t$ with coefficients $J_k\in L^\infty(U_i)$.
Since $J_0(s)=J(s,0)$ is given by \lae{e33} then $J(s,t)>0$ for all $|t|<\e$ and almost
all $s\in U_i$ provided $\e>0$ is conveniently reduced. In addition, equality \lae{e30} holds a. e. in
$U_i$. Then we deduce \lae{e31} and thus, after integration in $s$, we arrive to estimate
\lae{e32}.

On the other hand both the identities,
$$
\int_\Om \iz|\frac{\p u_i}{\p \bn}\der|\ dx = \int_{V_i} \iz|\frac{\p u_i}{\p \bn}\der|\ dx =
\int_{U_i}\int_0^{\e}|\p_t u_i(s,t )| J(s,t )\ dt \ ds,
$$
and
$$
\int_\Om |u_i|\frac{|J_t|}J\ dx = \int_{V_i} |u_i|\frac{|J_t|}J\ dx
=
\int_{U_i} \int_0^{\e}|u_i(s,t )|\frac{|\p_t J(s,t )|}{J(s,t )} J(s,t )\ dt \ ds,
$$
are achieved after performing the substitution $x=h_i(s,t)$ and employing the change
of variables formula, in its Lipschitz version \cite[Th. 3.3.2]{EG}.

This concludes the proof in the $C^{1,1}$ case.
\end{proof}

\begin{obs}\label{o:5}
\rm It should be stressed that the possibility of achieving $C_1=1$ in \lae{i1} is
suggested in \cite[Rem. 1.1]{Giu} by the terse quotation: ``{\it \dots if the
mean curvature of $\p\Om$ is bounded above then \lae{i1} holds with $C_1=1$}''.
\end{obs}

We are next exploiting the argument of the previous proof to show Theorem \ref{t2}.

\begin{proof}
[Proof of Theorem \ref{t2}] Unlike the case of Theorem \ref{t1}, the inner unit field $\bn$
is now only continuous. Anyway, a regularization $X\in C^\infty(\eren,\eren)$ of $\bn$ can been obtained
such that,
$$
|X(x)-\bn(x)| \le \eta,\qquad \text{for $x\in \uno{d(x)\le d_0}$},\qquad d_0>0,
$$
where $0 < \eta <1$ is a fixed small number. By denoting $x=\phi(t,y)$ the solution to,
$$
\frac{dx}{dt} = X(x),\qquad x(0)=y,
$$
i. e. $\phi$ is the flow associated to $X$, $t_0>0$ can be further found so that,
$$
|X(\phi(t,y))-\bn(y)| \le \eta,\qquad \text{for all $|t|\le t_0$ and each $y\in \p\Om$}.
$$
As in Theorem \ref{t1} we select a finite atlas $\uno{(g_i,U_i)}_{\rango 1im}$ for
$\p\Om$ and set the mappings $h_i(s,t) = \phi(t,g_i(s))$. They are diffeomorphisms
onto their images, now denoted as $W_i$, when such mappings are defined in
$U_i\times \uno{|t|<\e}$ with $\e>0$ small.

After properly picking an open $W_0\subset\subset \Om$ and considering a partition of unity
$\{\psi_i\}_{{\rango 0im}}$ associated to the covering $\uno{W_i}_{\rango 0im}$ of $\Omb$,
we decompose again,
$$
u = \sum_{i=0}^m \psi_i u = \sum_{i=0}^m u_i.
$$

As an alternative to \lae{e2} we are now obtaining an inequality of the form,
\begin{equation}
\label{e22}
\int_{\p \Om}|u_i|\le \frac{1+2\eta}{1-\eta}\int_\Om |Du_i|+ C_i'\int_\Om |u_i|,
\end{equation}
for all $\rango 1im$.

To begin with, we define,
$$
\tij(s,t) = \text{\rm det}\, \text{\rm col} (D\phi\, g_{s_1},\dots,D\phi \, g_{s_{N-1}},X),
$$
where $D\phi = D_y\phi$ and index $i$ in $g$ has been dropped. By taking into account that,
\begin{multline}\label{e240}
\frac{\tij(s,0)}{J(s,0)} =
\frac{{g_i}_{s_1}\wedge \cdots \wedge {g_i}_{s_{N-1}}}{|{g_i}_{s_1}\wedge \cdots \wedge {g_i}_{s_{N-1}}|}\cdot X(g(s))
%\\
= \bn(g(s))\cdot X(g(s))\ge 1-\eta,
\end{multline}
it follows in particular the positiveness of the factor $\tij(s,0)$.

We next mimick the corresponding step in Theorem \ref{t1} to obtain,
\begin{multline}\label{e23}
\int_{U_i} |u_i(s,0)|\tij(s,0)\ ds
\le \int_{U_i}\int_0^{\e_i}|\p_t u_i(s,t)| \tij(s,t)\ dt\ ds \\ +
\int_{U_i} \int_0^{\e_i}|u_i(s,t)|\frac{|\p_t \tij(s,t)|}{\tij(s,t)}\tij(s,t)\ dt\ ds,
\end{multline}
where $u_i(s,t)= u(\phi(t,g(s)))$ and $g=g_i$.
Thus, by employing \lae{e240} we are lead to,
\begin{equation}
\label{e24}
(1-\eta)\int_{\p\Om}|u_i| \le \int_{U_i} |u_i(s,0)|\tij(s,0)\ ds.
\end{equation}

Next, one observes that for every $y\in\p\Om$,
$$
|X(\phi(t,y))|\le |X(\phi(t,y))-X(y)|+|X(y)-\bn(y)|+1\le 2\eta +1.
$$
Hence,
$$
|\p_t u_i(s,t)|\le (1+2\eta)|Du_i|.
$$
Accordingly,
\begin{equation}
\label{e25}
\int_{U_i}\int_0^{\e_i}|\p_t u_i(s,t)| \tij(s,t)\ dt\ ds\le (1+2\eta)\int_\Om |Du_i|.
\end{equation}
By combining \lae{e23}, \lae{e24} and \lae{e25} we arrive to the desired inequality \lae{e22}.
Finally a reasoning as the one in Theorem \ref{t1} leads us to the estimate \lae{i1} with
the constant,
$$
C_1 = \frac{1+2\eta}{1-\eta}.
$$
It is clear that $C_1$ can be taken as close to unity as desired and so we are done.
\end{proof}

The proof of Lemma \ref{l1} is now performed in a $C^1$ domain by forming a smooth
family of functions $u_n$. Indeed, let $\zeta\in C^\infty(\R)$ be so that $0\le \z \le 1$, $\z(0)=1$,
$\z'<0$ for $0< t < 1$ and $\z=0$ in $t\ge 1$. We define,
$$
u_n(x) = \sum_{i=1}^m \psi_i(x)\z \iz ( \frac{t(x)}{\e_n}\der),\qquad \e_n\to 0+,
$$
where $x\to (s(x),t(x))$ is the inverse mapping of $h_i(s,t)$ when regarded in $W_i$ (families
$\uno{W_i}$ and $\uno{\psi_i}$ are the ones introduced in the previous proof of Theorem \ref{t2}).
We notice that $u_n\in C^1(\Omb)$ while,
$$
\int_{\p\Om}|u_n| = |\p\Om|_{N-1}\, \qquad\text{and}\qquad \int_{\Om}|u_n| = o(1),\qquad
\text{as $n\to \infty$.}
$$
On the other hand,
$$
\int_{\Om}|Du_n|\le
\sum_{i=1}^m \int_{\Om}\e_n^{-1}\psi_i \iz|\z'\iz (\frac{t}{\e_n}\der )
\na t \der| + o(1),\qquad
\text{as $n\to \infty$.}
$$
Now observe that,
\begin{multline*}
\int_{\Om}\e_n^{-1}\psi_i \iz|\z'\iz (\frac{t}{\e_n}\der )
\na t \der|
%\\
= -\int_0^1 \int_{U_i}\z'(t )\psi_i(s,\e_nt )|\na t(s,\e_nt )|\tij(s,\e_nt )
\\
\to \int_{U_i}\psi_i(s,0)|\na t(s,0)|\tij(s,0),\qquad \text{as $n\to \infty$.}
\end{multline*}
In addition,
$$
\tij(s,0) = \bn(g(s))\cdot X(g(s)) J(s,0),
$$
while a direct computation yields,
$$
\na t(s,0)= \frac 1{\bn(g(s))\cdot X(g(s))} \bn(g(s)).
$$
Therefore,
$$
\varlimsup \int_{\Om}|Du_n|\le \sum_{i=1}^m \int_{U_i}\psi_i(s,0) J(s,0) = |\p\Om|_{N-1}.
$$
Inequality $C_1\ge 1$ is then obtained by setting $u=u_n$ in \lae{i1} and passing to the limit.

\section{Some examples}\label{s:4}

Our first example shows that the conclusion of Theorem \ref{t2} is in general false
if $\Om$ is a piecewise $C^1$ domain (cf. also \cite[p. 12]{AnG}).

\begin{ejem}
\label{eje0} \rm
Let $\Om$ be an open domain such that for a certain small $\eta >0$ its shape near $0$
is described as,
$$
\Om\cap \uno{x_N<\eta} = \left\{x= (x',x_N): L|x'| < x_N<\eta \right\},
$$
where here and in the subsequent examples we are denoting $x'=(x_1,\dots,x_{N-1})$.
Assume also that $\p\Om\cap\uno{x_N>\frac \eta2}$ is a smooth hypersurface so that
$\Om$ is piecewise $C^1$.

We set $F_n = \uno{L|x'|< x_N<L r_n,\, |x'|<r_n}$, $u_n = \chi_{F_n}$ (the characteristic
function of $F_n$), where $L>0$ and are assuming that $r_n\to 0+$. Then,
$$
\int_{\p\Om} |u_n| = \sqrt{1+L^2}\,\, \om_{N-1}r_n^{N-1},
$$
$\om_{N-1}$ being the volumen of $\uno{|x'|<1}$. In addition,
$$
\int_{\Om}|Du_n| = P(E_n,\Om) = \om_{N-1}r_n^{N-1},
$$
where the gradient of $u_n$ has been performed in distributions sense (Section 2.1).
Since,
$$
\int_\Om |u_n| = \frac{\om_{N-1}L}N r_n^N,
$$
then constant $C_1$ in \lae{i1} is constrained by the inequality,
$$
C_1\ge \sqrt{1+L^2}.
$$
Thus $C_1$ is bounded away from unity.
\end{ejem}

Next example shows that an inequality as \lae{e:ppal} fails to be true if the smoothness
of $\p\Om$ is somehow weakened from $C^{1,1}$. In fact, it proves that $C^{1,\al}$ domains,
$0 < \al <1$, are not sufficiently smooth to ensure us the validity of \lae{e:ppal}.

\begin{ejem}\label{eje1}
\rm
Let $\Om\subset \R^{N}$ be an open $C^{1,\al}$ domain with $0\in \p\Om$, satisfying,
$$
\Om\cap \uno{x_N<\eta} = \left\{x= (x',x_N): x_N>\frac{|x'|^{1+\al}}{1+\al}\right\},\qquad 0 <\al <1,
$$
for a certain positive and small enough number $\eta$. Fix
a sequence $r_n\to 0+$ and define for large $n$,
$$
E_n =\uno{\psi(|x'|)< x_N<\psi(r_n)}\subset \Om, \qquad \psi(t)=\frac{|t|^{1+\al}}{1+\al}.
$$
We are showing that the family $u_n = \chi_{E_n}$
cannot satisfy the inequality \lae{e:ppal} when $n\to \infty$. In fact,
\begin{multline*}
\int_{\p\Om}|u_n| = \int_{|x'|< r_n}\sqrt{1+|x'|^{2\al}}\ dx'\\
= (N-1)\om_{N-1}\iz (\frac{r_n^{N-1}}{N-1}+\frac 12\frac{r_n^{N-1+2\al}}{N-1 + 2\al}+ O(r_n^{N-1 +4\al})\der)\\
= \om_{N-1}{r_n^{N-1}}+\frac{(N-1)\om_{N-1}}{2(N-1+2\al)}r_n^{N-1+2\al}+ O(r_n^{N-1+4\al}).
\end{multline*}
 As in Example \ref{eje0},
$$
\int_{\Om}|Du_n| = P(E_n,\Om) = \om_{N-1}r_n^{N-1}.
$$
Thus,
$$
\int_{\p\Om}|u_n| - \int_{\Om}|Du_n| =
\frac{(N-1)\om_{N-1}}{2(N-1 +2\al)}r_n^{N-1+ 2\al}+ O(r_n^{N-1 +4\al}).
$$
In addition,
$$
\int_\Om |u_n| = |E_n| = \frac{\om_{N-1}}{N+\al}r_n^{N+\al}.
$$
Since $0 <\al <1$ exponent $N+\al$ is greater than $N-1+2\al$.
Therefore, inequality \lae{e:ppal} cannot be satisfied when
$r_n\to 0+$.
\end{ejem}

\begin{obs}
\rm Let us check that the normal curvatures of $\p\Om$ blow up
when we approach to $0$ and the boundary is oriented near $0$ with the
inner normal field,
$$
\bla n = \frac1{\sqrt{1+|x'|^{2\al}}}\iz\{\der(-|x'|^{\al-1}x',0)+e_N\}.
$$
In fact, if $v\in \R^{N-1}$ is an unitary
vector, the curve in $\p\Om$ defined as
$$
x = t(v,0) + \frac{|t|^{\al +1}}{\al +1}e_N,\qquad |t|<\e,
$$
passes through $0$ when $t=0$. It is in fact a normal section to $\p\Om$ since each of its
points is contained in the plane $\text{\rm span\,} \uno{(v,0),e_N}$, while
the family of normals $\bla n(x(t))$ also lies in this plane. The expression,
\begin{equation}\label{e4:1}
\kappa = \frac 1{|\dot x|^2}\iz |\ddot x - (\ddot x.\bla t)\bla t\der| = \frac{\al |t|^{\al-1}}{{\{\sqrt{1+t^{2\al}}}\}^3},
\qquad \bla t = \frac{\dot x}{|\dot x|},
\end{equation}
where `dots' mean differentiation with respect to $t$,  provides us the normal curvature of $\p\Om$ at $x(t)$ and the
direction $v$. It can be also checked that $\kappa$ is the minimum of the principal curvatures at this point. Then, it
follows from \lae{e4:1} that all normal curvatures diverge to infinity as $t\to 0+$.
This also occurs  to the mean curvature $H$. Thus, Example \ref{eje1} is coherent with the suggestion in
\cite{Giu} quoted in Remark \ref{o:5}.
\end{obs}

\begin{obs}

\label{o:6}
\rm
Example \ref{eje1} also provides us a $C^{1,\al}$ manifold $\p\Om$ with $0 < \al<1$ where $r_{\p\Om}=0$.
In fact we are proving that $r(0)=0$. To this proposal we take $p_\s = (0,\s)\in \R^{N-1}\times \R^+$ and claim
that the sphere $|x-p_\s|=\s$ possesses points outside $\Omb$ for every $\s>0$ small.
This would imply that $\text{dist\,}(p_\s,\p\Om)<\s$ and so such a distance should be
achieved at infinitely many points $y\in\p\Om\menos \uno{0}$
symmetrically distributed around the $x_N$ axis.

To show the claim it is enough with comparing the profiles of the sphere
$x_N = \s-\sqrt{\s^2-|x'|^2}=:\phi(|x'|)$ and the boundary $x_N = \psi(|x'|)$ of $\p\Om$
near the origin. It is clear that $\phi(t)<\psi(t)$ as $t\to 0+$ while $\phi(\s)=\s>\psi(\s)$
when $0<\s <1$. Thus, the checking is complete.

As an extra remark, it follows from \cite[Th. 4.8--3)]{Fe} that the distance function
$d(x)=\text{dist\,}(x,\p\Om)$ can not be differentiable at any of the points $p_\s$.
In other words, $d(x)$ does not define a $C^1$ function under the conditions of the
present example (cf. \cite{Krap}).
This should be compared with the domains $\Om$ of class $C^{2}$ where the signed distance
$\tilde d(x)$ defines a $C^2$ function in a neighborhood of $\p\Om$ (\cite[Chap. 14]{GT}, \cite{Krap}).
\end{obs}

Finally, let us consider the functionals of type \lae{i3},
$$
J_\pm(u)=\int_\Om |Du| \pm \int_{\p\Om}|u|,
$$
in the space $BV(\Om)$. It is well--known that $J_+$ is lower semicontinuous
in $BV(\Om)$ with respect to the $L^1$ convergence
if $\Om$ is only a bounded Lipschitz domain (\cite[Rem. 1.3]{Mod}).
In strong contrast, we are next showing that $J_-$ fails to exhibit this behavior even
in a more smooth domain.

\begin{ejem}\label{e:mod}
\rm Let $\Om$ be the  $C^{1,\al}$ domain introduced in Example \ref{eje1}.
Let $u_n= \chi_{E_n}$ also be the sequence of functions considered there.
We now set,
$$
v_n = A_n\chi_{E_n},\qquad A_n = \frac 1{r_n^{N-1+2\al}}.
$$
Then,
\begin{multline*}
J_-(v_n)= A_n J_-(u_n)
%\\
= A_n\iz(\int_\Om |Du_n| - \int_{\p\Om}|u_n|\der)
\\
 =-\frac{(N-1)\om_{N-1}}{2(N-1 +2\al)}+ O(r_n^{2\al}) < 0,
\end{multline*}
for large $n$. However,
$$
\|v_n\|_{L^1(\Om)} = A_n |E_n| = \frac{\Om_{N-1}}{N+\al}r_n^{1-\al}\to 0.
$$
Therefore, the conclusion of Theorem \ref{t3} can not be true in this case.
\end{ejem}

\subsection*{Acknowledgements}

This research has been partially supported by
Grant PID2022-136589NB-I00 founded by MCIN/AEI/10.13039/50\-1100011033.

\subsection*{Declarations. Conflict of interest} Author declares that has no conflict of interest.

\end{document}